\documentclass[a4, 12pt]{amsart}
\usepackage{amssymb}
\usepackage{amstext}
\usepackage{amsmath}
\usepackage{amscd}
\usepackage{latexsym}
\usepackage{amsfonts}

\theoremstyle{plain}
\newtheorem{thm}{Theorem}[section]
\newtheorem*{thm*}{Theorem}
\newtheorem*{cor*}{Corollary}

\newtheorem{prop}[thm]{Proposition}
\newtheorem{lem}[thm]{Lemma}
\newtheorem{cor}[thm]{Corollary}
\newtheorem{claim}{Claim}
\newtheorem*{claim*}{Claim}

\theoremstyle{definition}

\newtheorem{ex}[thm]{Example}

\theoremstyle{remark}

\numberwithin{equation}{thm}

\def\mod{\rm{mod}}

\def\rank{\rm{rank}}
\def\a{\mathfrak a}

\def\e{\mathrm e}
\def\m{\mathfrak m}

\def\p{\mathfrak p}

\def\Z{\Bbb Z}

\def\H{\rm{H}}

\def\Gr{\mathrm G}

\newcommand{\rma}{\rm{a}}

\newcommand{\rme}{\rm{e}}

\newcommand{\rmH}{\rm{H}}

\newcommand{\fkm}{\mathfrak{m}}

\newcommand{\fkp}{\mathfrak{p}}

\def\depth{\rm{depth}}

\def\Ass{\rm{Ass}}

\def\height{\rm{ht}}

\begin{document}

\setlength{\baselineskip}{20pt}

\title{Sally modules of rank one}
\author{Shiro Goto}
\address{Department of Mathematics, School of Science and Technology, Meiji University, 1-1-1 Higashi-mita, Tama-ku, Kawasaki 214-8571, Japan}
\email{goto@math.meiji.ac.jp}
\author{Koji Nishida}
\address{Department of Mathematics and Informatics, Graduate School of Science and Technology, Chiba University, 1-33 Yatoi-cho, Inage-ku, Chiba-shi, 263 Japan}
\email{nishida@math.s.chiba-u.ac.jp}
\author{Kazuho Ozeki}
\address{Department of Mathematics, School of Science and Technology, Meiji University, 1-1-1 Higashi-mita, Tama-ku, Kawasaki 214-8571, Japan}
\email{kozeki@math.meiji.ac.jp}
\thanks{{\it Key words and phrases:}
Cohen-Macaulay local ring, Buchsbaum ring, associated graded ring, Rees algebras,
Sally module, Hilbert coefficients.
\endgraf
{\it 2000 Mathematics Subject Classification:}
13H10, 13B22, 13H15.}
\maketitle
\begin{abstract}
The structure of Sally modules of $\fkm$-primary ideals $I$ in a Cohen-Macaulay local ring $(A, \m)$ satisfying the equality $\e_1(I)=\e_0(I)-\ell_A(A/I)+1$  is explored, where $\e_0(I)$ and $\e_1(I)$ denote the first two Hilbert coefficients of $I$.
\end{abstract}
\section{Introduction}
Let $A$ be a Cohen-Macaulay local ring with the maximal ideal $\m$ and $d=\dim A >0$. We assume the residue class field $k=A/\fkm$ of $A$ is infinite. Let $I$ be an $\fkm$-primary ideal in $A$ and choose a minimal reduction $Q=(a_1, a_2, \cdots, a_d)$ of $I$. Let 
$$R = \mathrm{R}(I) := A[It]~~\ \ \operatorname{and}~~\ \ T= \mathrm{R}(Q):= A[Qt]~~\subseteq~~A[t]$$ respectively denote the Rees algebras of $I$ and $Q$, where $t$ stands for an indeterminate over $A$. We put $$R' = \mathrm{R}'(I) := A[It, t^{-1}], ~~~~~\ \ T' = \mathrm{R}'(Q):= A[Qt, t^{-1}],$$
and $$G= \mathrm{G}(I) :=  R'/t^{-1}R'~~\cong~~\bigoplus_{n \geq 0}I^n/I^{n+1}.$$ 
Let $B=T/\m T$ which is the polynomial ring with $d$ indeterminates over the field $k$. Following W. V. Vasconcelos \cite{V}, we then define
 $$\mathrm{S}_Q(I)= IR/IT$$ and call it the Sally module of $I$ with respect to $Q$. We notice that the Sally module $S=\mathrm{S}_Q(I)$ is a finitely generated graded $T$-module, since $R$ is a module-finite extension of the graded ring $T$.

Let $\ell_A(*)$ stand for the length and consider the Hilbert function $$H_I(n)= \ell_A(A/I^{n+1})$$
$(n \geq 0)$  of $I$. Then we have the integers $\{\e_i=\e_i(I)\}_{0 \leq i \leq d}$ so that the equality
$$H_I(n)={\rme}_0\binom{n+d}{d}-{\e_1}\binom{n+d-1}{d-1}+\cdots+(-1)^d{\e}_d$$
holds true for all $n \gg 0$.

The Sally module $S$ was introduced by W. V. Vasconcelos \cite{V}, where he gave an elegant review, in terms of his $Sally$ module, the works \cite{S1, S2, S3}  of J. Sally about the structure of $\m$-primary ideals $I$ with interaction to the structure  of $G$ and Hilbert coefficients $\e_i$'s. J. Sally  firstly investigated those ideals $I$ satisfying the equality $\e_1 = \e_0 - \ell_A(A/I) + 1$ and gave several  very important results, among which one can find the following characterization of ideals $I$ with $\e_1 = \e_0 - \ell_A(A/I) + 1$ and $\e_2 \ne 0$, where $B(-1)$ stands for the graded $B$-module whose grading is given by $[B(-1)]_n = B_{n-1}$ for all $n \in \Z$. 
The reader may also consult with \cite{CPVP} and \cite{VP} for further  ingenious use of Sally modules.

\begin{thm}[Sally \cite{S3} , Vasconcelos \cite{V}]\label{Sally}
The following three conditions are equivalent to each other.
\begin{itemize}
\item[$(1)$] $S \cong B(-1)$ as graded $T$-modules.
\item[$(2)$] $\e_1=\e_0-\ell_A(A/I)+1$ and if $d \geq 2$, $\e_2 \ne 0$.
\item[$(3)$] $I^3 = QI^2$ and $\ell_A(I^2/QI) = 1$.
\end{itemize}
When this is the case, the following assertions hold true.
\begin{itemize}
\item[$(i)$] $\e_2 = 1$, if $d \geq 2$.
\item[$(ii)$] $\e_i = 0$ for all $3 \leq i \leq d$, if $d \geq 3$.
\item[$(iii)$] $\operatorname{depth} G \geq d-1$.
\end{itemize}
\end{thm}

The present research is  a continuation of  \cite{S3, V} and aims at similar  understanding of the structure of Sally modules of ideals $I$ which satisfy the equality 
$\e_1=\e_0- \ell_A(A/I)+1$ but $\e_2 = 0$. When $\m S=(0)$, we denote by $\mu_B(S)$ the number of elements in a minimal homogeneous system of generators of the graded $B$-module $S$. Let $$\tilde{I}= \bigcup_{n \geq 1}[I^{n+1} : I^n]=\bigcup_{n \geq 1}[I^{n+1}:(a_1^n, a_2^n, \cdots, a_d^n)]$$ denote the Ratliff-Rush closure of $I$ (cf. \cite{RR}), which is the largest $\m$-primary ideal of $A$ such that $I \subseteq \tilde{I}$ and $$\e_i(\tilde{I}) = \e_i(I)~~\operatorname{for~all}~0 \leq i \leq d.$$ With this notation the main result of this paper is stated as follows.

\begin{thm}\label{MainTheorem}
Suppose that $d \geq 2$. Then the following four conditions are equivalent to each other.
\begin{itemize}
\item[$(1)$] $\m S = (0)$, $\operatorname{rank}_BS =1$, and $\mu_B(S) = 2$. 
\item[$(2)$] There exists an exact sequence 
$$0 \to B(-2) \to B(-1) \oplus B(-1) \to S \to 0$$of graded $T$-modules.
\item[$(3)$] $\e_1 = \e_0 - \ell_A(A/I) + 1$, $\e_2 =0$, and $\operatorname{depth} G \geq d-2$.
\item[$(4)$] $I^3 =QI^2$, $\ell_A(I^2/QI) = 2$, $\m I^2 \subseteq QI$, and $\ell_A(I^3/Q^2I) < 2d$.
\end{itemize}
When  $d=2$, one can add the following condition$\mathrm{:}$
\begin{itemize}
\item[$(5)$] $\ell_A(\tilde{I}/I) = 1$ and $\tilde{I}^2 = Q\tilde{I}$.
\end{itemize}
When one of  conditions $(1), (2), (3)$, and $(4)$ is satisfied, the following assertions hold true
\begin{itemize}
\item[$(i)$] $\operatorname{depth} G = d-2$,
\item[$(ii)$] $\e_3 = -1$, if $d \geq 3$,
\item[$(iii)$] $\e_i = 0$ for all $4 \leq i \leq d$, if $d \geq 4$,
\item[$(iv)$] $\ell_A(I^3/Q^2I) = 2d - 1$,
\end{itemize}
and, when $d=2$ and condition $(5)$ is satisfied, the graded rings $G$, $R$, and $R'$ are all Buchsbaum rings with the same Buchsbaum invariants $$\Bbb{I} (G) = \Bbb{I}(R)= \Bbb{I}(R') = 2.$$
\end{thm}

Combined with Theorem \ref{Sally}, this theorem gives, in the case where $d = 2$, a complete structure theorem of Sally modules of those ideals $I$ with  $\e_1 = \e_0 - \ell_A(A/I) + 1$ (cf. Theorem \ref{d=2}). We could similarly describe the structure of Sally modules in higher dimensional cases also, if one could show that $I^3=QI^2$ if $\e_1 = \e_0 - \ell_A(A/I) + 1$, which we surmise holds true, although we could not prove the implication.

Let us now briefly explain how this paper is organized. We shall prove Theorem \ref{MainTheorem} in Section 3. The key for our proof of Theorem \ref{MainTheorem}  is Theorem  \ref{MainTheoremA}, whose applications we will closely discuss in Section 2.  
Section 2 is devoted also to some auxiliary facts on Sally modules, some of which  are more or less known but we shall indicate brief proofs for the sake of completeness. If $\e_1=2$ but $I^2\ne QI$, the ideal $I$ naturally satisfies the equality $\e_1 = \e_0 - \ell_A(A/I) + 1$. In Section 4 we shall  explore those ideals $I$ with $\e_1 = 2$ but $I^2 \ne QI$, in connection with the Buchsbaum property of the graded rings $R, G, $ and $R'$ associated to $I$. We shall explore in Section 5 one example in order to illustrate our theorems.

In what follows, unless otherwise specified, let  $(A, \m)$ be a Cohen-Macaulay local ring with $d=\dim A >0$. We assume that the field $A/ \m$ is infinite. Let $I$ be an $\m$-primary ideal in $A$ and let $S$ be the Sally module of $I$ with respect to a minimal reduction $Q = (a_1, a_2, \cdots, a_d)$ of $I$. We put $R= A[It], T=A[Qt], R' = A[It, t^{-1}]$, $T' = A[Qt, t^{-1}]$, and $G= R'/t^{-1}R'$. Let $M = \m T + T_+$ be the unique graded maximal ideal in $T$. We denote by ${\H}_M^i(*)$~$(i \in \Z)$ the $i^{\underline{th}}$ local cohomology functor of $T$ with respect to $M$. Let $L$ be a graded $T$-module. For each $n \in \Z$ let $[{\H}_M^i(L)]_n$ stand for the homogeneous component of ${\H}_M^i(L)$ with degree $n$. We denote by $L(\alpha)$, for each $\alpha \in \Z$, the graded $T$-module whose grading is given by $[L(\alpha)]_n = L_{\alpha +n}$ for all $n \in \Z$.


\section{Preliminaries}

The purpose of this section is to summarize some auxiliary results on Sally modules, which we will use throughout this paper. Some of the results are known but let us include brief proofs for the sake of completeness.

We begin with the following.

\begin{lem}\label{lemma1}
The following assertions hold true.
\begin{itemize}
\item[$(1)$] $\m^{\ell} S = (0)$ for  integers $\ell \gg 0$.
\item[$(2)$]The homogeneous components $\{ S_n \}_{n \in \Z}$ of the graded $T$-module $S$ are given by
\[ S_n \cong  \left\{
\begin{array}{rl}
(0) & \quad \mbox{if $n \leq 0 $,} \\
I^{n+1}/IQ^n & \quad \mbox{if $n \geq 1$.}
\end{array}
\right.\]
\item[$(3)$] $S=(0)$ if and only if $I^2=QI$.
\item[$(4)$] Suppose that $S \ne (0)$ and put $V = S/MS$. Let $V_n$~$(n \in \Z)$  denote the homogeneous component of the finite-dimensional graded $T/M$-space $V$ with degree $n$ and put $\Lambda = \{n \in \Z \mid V_n \ne (0) \} $. Let $q = \max \Lambda$. Then we have $\Lambda =\{1, 2, \cdots, q\}$ and $\mathrm{r}_Q(I) = q+1$, where $\mathrm{r}_Q(I)$ stands for  the reduction number of $I$ with respect to $Q$.   
\item[$(5)$] $S = TS_1$ if and only if $I^3 = QI^2$. 
\end{itemize}
\end{lem}

\begin{proof}
Let $u = t^{-1}$ and notice that $S = IR/IT \cong IR'/IT'$ as graded $T$-modules. We then have $u^{\ell}{\cdot}(IR'/IT') = (0)$ for some $\ell \gg 0$, because the graded $T'$-module $IR'/IT'$ is finitely generated and $[IR'/IT']_n =(0)$ for all $n \leq 0$.  Hence $\m^{\ell}{\cdot}S = (0)$ for $\ell \gg 0$, because $Q^{\ell} = (Qt^{\ell})u^{\ell} \subseteq u^{\ell}T' \cap A$ and $\m = \sqrt{Q}$. This proves assertion (1).

Since $[IR]_n = (I^{n+1})t^n$ and $[IT]_n = (IQ^n)t^n$ for all $n \geq 0$, assertion (2) follows from the definition of the Sally module $S = IR/IT$. Assertion (3) readily follows from assertion (2).

To show assertion (4), notice that $V_1 \cong S_1/\m S_1 \ne (0)$, since $S =\sum_{n \geq 1}S_n$ and $S_1 \cong I^2/QI \ne (0)$. Hence $1 \in \Lambda$. Let $i \in \Lambda$ and put $\alpha_i = \dim_kV_i$, where $ k =T/M$. We choose elements $\{\xi_{i, j}\}_{1 \leq j \leq \alpha_i}$ of $S_i$ so that the images of $\{\xi_{i, j}\}_{1 \leq j \leq \alpha_i}$ in $V$ form a $k$-basis of $V_i$. Hence, thanks to graded Nakayama's lemma, we have 
$$S =\sum_{i \in \Lambda}(\sum_{j=1}^{\alpha_i}T\xi_{i,j}).$$ Let $\xi_{i,j}$ be the image of $x_{i,j}t^{i}$ in $S$ with $x_{i,j} \in I^{i+1}$.

Let $n \geq 1$ be an integer  and assume that $n \not\in \Lambda$.  Choose $x \in I^{n+1}$ and let $\xi$ be the image of $xt^n$ in $S$. We write 
$$\xi =\sum_{i \in \Lambda , i < n}(\sum_{j=1}^{\alpha_i}\varphi_{i,j}\xi_{i,j})$$ with $\varphi_{i,j} \in T_{n-i}$.  Then, letting $\varphi_{i,j} = b_{i,j}t^{n-i}$ with $b_{i,j} \in Q^{n-i}$, we get $$x \equiv \sum_{i \in \Lambda , i < n}(\sum_{j=1}^{\alpha_i}b_{i,j}x_{i,j})~~{\mod}~~Q^nI, $$ whence $x \in QI^n$, because $\sum_{j=1}^{\alpha_i}b_{i,j}x_{i,j} \in Q^{n-i}I^{i+1} \subseteq QI^{n}$ for all $i \in \Lambda$ such that $i < n$. Thus $I^{n+1} = QI^n$. Suppose now $n \leq q$. Then $I^{q+1}= QI^q$, whence $S_q \subseteq T_+S$ and so $V_q = (0)$, which is impossible. Hence $\Lambda = \{1, 2, \cdots, q\}$. Choosing $n = q+1$, the above observation shows that $I^{q+2} = QI^{q+1}$, whence $\mathrm{r}_Q(I) \leq q+1$. If $r = \mathrm{r}_Q(I) < q+1$, we have $I^{q+1} =QI^q$, whence $S_q \subseteq T_+S$, which is absurd. Thus $\mathrm{r}_Q(I) = q+1$. This proves assertion (4). Assertion (5) is now clear. 
\end{proof}

\begin{prop}\label{proposition2}
Let $\fkp = \m T$. Then the following assertions hold true.
\begin{itemize}
\item[$(1)$] ${\Ass}_TS \subseteq \{ \fkp \}$. Hence $\dim_TS=d$, if $S\ne (0)$. 
\item[$(2)$] $\ell_A(A/I^{n+1})={\e}_0\binom{n+d}{d}-({\e}_0-\ell_A(A/I)){\cdot}\binom{n+d-1}{d-1}-\ell_A(S_n)$ for all $n \geq 0$.
\item[$(3)$] We have $\e_1 = \e_0 - \ell_A(A/I) + \ell_{T_\fkp}(S_\fkp)$. Hence
$\e_1 = \e_0 - \ell_A(A/I) + 1$ if and only if $\m S = (0)$ and $\operatorname{rank}_BS = 1$. 
\item[$(4)$] Suppose that $S\ne (0)$. Let $s = \operatorname{depth}_TS$. Then  $\operatorname{depth} G = s-1$ if $s < d$. $S$ is a Cohen-Macaulay $T$-module if and only if  $\operatorname{depth} G \geq d-1$.
\end{itemize}
\end{prop}

\begin{proof}
(1) Let $P \in {\Ass}_TS$. Then $\fkp = {\m} T \subseteq P$, since $\m^\ell S=0$ for some $\ell \gg 0$ (Lemma \ref{lemma1} (1)). Since $\operatorname{ht}_T\fkp=1$, it is enough to show that $\operatorname{ht}_TP \leq 1$. We look at the exact sequence 
$$0 \to IT_{P} \to IR_{P} \to S_{P} \to 0$$ of $T_P$-modules 
and recall that $IT$ is a Cohen-Macaulay $T$-module with $\dim_TIT = d+1$, because $$T/IT = (A/I)\otimes_{A/Q}(T/QT)$$ is the polynomial ring with $d$ indeterminates over $A/I$ and $T$ is a Cohen-Macaulay ring with $\dim T = d+1$. Notice now that $a_1 \in P$ is a nonzerodivisor on $IR$, whence $\operatorname{depth}_{T_P}IR_P > 0$. Thanks to depth lemma, it follows from the above exact sequence that $\dim_{T_P}{IT_P} = 1$, since $\operatorname{depth}_{T_P}IR_P >0$ and $\operatorname{depth}_{T_P}S_P = 0$. Hence $\dim T_P = 1$, because $IT$ is a Cohen-Macaulay $T$-module with $(0) :_TIT =(0)$.
Thus $P = \fkp$ so that we have  ${\Ass}_TS = \{\fkp \}$ as is claimed.

(2) 
Let $n \geq 0$ be an integer. Then, thanks to the exact sequence
$$0 \to S_n \to A/Q^nI \to A/I^{n+1} \to 0$$
of $A$-modules (Lemma \ref{lemma1} (2)), we have
$$\ell_A(A/I^{n+1})=\ell_A(A/Q^nI)-\ell_A(S_n),$$
while by  the exact sequence
$$0 \to Q^n/Q^nI \to A/Q^nI \to A/Q^n \to 0$$ we get
\begin{eqnarray*}
\ell_A(A/Q^nI)&=&\ell_A(A/Q^n)+\ell_A(Q^n/Q^nI)\\
&=&\ell_A(A/Q) {\cdot}\binom{n+d-1}{d}+\ell_A(Q^n/Q^nI)\\
              &=&\e_0 \binom{n+ d -1}{d}+\ell_A(Q^n/Q^nI)\\
              &=&\e_0 \binom{n+d}{d} - \e_0 \binom{n+d-1}{d-1} + \ell_A(Q^n/Q^nI),
\end{eqnarray*}
because $\e_0 = \ell_A(A/Q)$ (recall that $Q=(a_1, a_2, \cdots, a_d)$ is a minimal reduction of $I$).
Thanks to the isomorphisms
$$Q^n/Q^nI \cong (A/I) \otimes_A (Q^n/Q^{n+1}) \cong (A/I) \otimes_A[(A/Q)^{\binom{n+d-1}{d-1}} ]\cong (A/I)^{\binom{n+d-1}{d-1}},$$
we furthermore have the equality
$$\ell_A(Q^n/Q^nI)=\ell_A(A/I){\cdot}\binom{n+d-1}{d-1}.$$
Thus 
\begin{eqnarray*}
\ell_A(A/I^{n+1})&=&\ell_A(A/Q^nI)-\ell_A(S_n)\\
&=&[\e_0 \binom{n+d}{d} - \e_0 \binom{n+d-1}{d-1} + \ell_A(Q^n/Q^nI)]-\ell_A(S_n)\\              &=&[\e_0 \binom{n+d}{d} - \e_0 \binom{n+d-1}{d-1}+\ell_A(A/I){\cdot}\binom{n+d-1}{d-1}]-\ell_A(S_n)\\
              &=&\e_0 \binom{n+d}{d} - (\e_0 - \ell_A(A/I)){\cdot}\binom{n+d-1}{d-1}  -\ell_A(S_n)
\end{eqnarray*}
for all $n \geq 0$.

(3) If $S=(0)$, then $\e_1 = e_0 - \ell_A(A/I)$ by assertion (2). So, we may assume that $S \ne (0)$. We take a  filtration
$$ S= L_0 \supsetneq L_1 \supsetneq \cdots \supsetneq L_q=(0)$$
of the graded $T$-module $S$ such that each $L_i$ is a graded $T$-submodule of $S$ and 
$$L_i/L_{i+1} \cong (T/P_i)(-\alpha_i)$$
with some integer $\alpha_i$ for all $0 \leq i < q$,
where $P_i$ is a graded prime ideal of $T$.
Then, because ${\Ass}_TS =\operatorname{Min}_TS= \{\fkp \}$, we see that $\fkp \subseteq P_i$ for all $0 \leq i < q$. We furthermore have  $$\ell_{T_{\fkp}}(S_{\fkp}) =\sharp\{\mbox{ $i$ $\mid$ $0 \leq i < q$, $\fkp =P_i$} \},$$
since
$$\ell_{T_\fkp}(S_\fkp)=\sum_{i=0}^{q-1}\ell_{T_\fkp}((L_{i}/L_{i+1})_\fkp)=\sum_{i=0}^{q-1}\ell_{T_\fkp}(T_\fkp/P_i T_\fkp)$$
and
$$ T_\fkp/P_iT_\fkp =  \left\{
\begin{array}{ll}
B_\fkp & \quad \mbox{if $\fkp=P_i$}\\
(0) & \quad \mbox{if $ \fkp \subsetneq P_i$}.
\end{array}
\right.$$
On the other hand, we have
$$ \ell_{A}(S_n)=\sum_{i=0}^{q-1}\ell_A([L_i/L_{i+1}]_n)=\sum_{i=0}^{q-1}\ell_A([(T/P_i)(-\alpha_i)]_{n})$$
for all $n \in \Z$.
When $\fkp = P_i$, we get 
\begin{eqnarray*}
\ell_A([(T/P_i)(-\alpha_i)]_{n}) &=& \ell_A(B_{n - \alpha_i})=\binom{n - \alpha_i+ d-1}{d-1}\\
                             &=& \binom{n + d-1}{d-1} - \alpha_i \binom{n + d-2}{d-2}+ \mbox{(lower terms)}
\end{eqnarray*}
and when  $\fkp \subsetneq P_i$, we have ${\dim}\,T/P_i < d$,
so that  the degree of the Hilbert polynomial of $T/P_i$ is less than $d-1$.
Consequently, the normalized coefficient in degree $d-1$ of the Hilbert polynomial of the graded $T$-module $S$ is exactly equal to $\ell_{T_{\fkp}}(S_{\fkp})$ so that, thanks to assertion (2), we get the equality $\e_1 = \e_0 - \ell_A(A/I) + \ell_{T_{\fkp}}(S_{\fkp})$.

To see the second assertion, recall that ${\Ass}_TS = \{\fkp \}$. If $\ell_{T_{\fkp}}(S_{\fkp})=1$, then $\fkp S_{\fkp}=(0)$, so that $\fkp S =(0)$; hence $\m S = (0)$ and $\operatorname{rank}_BS =\ell_{T_{\fkp}}(S_{\fkp})= 1$. The reverse implication is clear.

(4) Recall that $s \leq d = \dim_TS$.  Because $IT$ is a Cohen-Macaulay $T$-module with $\operatorname{dim}_TIT = d + 1$, by the exact sequence $$0 \to IT \to IR \to S \to 0 \leqno{(a)}$$ we have $\operatorname{depth}_TIR \geq d$ if $s = d$ and $\operatorname{depth}_TIR = s$ if $ s < d$,  thanks to depth lemma. We put  $L = R_+$ and notice that $IR \cong L(1)$ as graded $R$-modules. Therefore, since $A$ is a Cohen-Macaulay ring with $\dim A = d$, by the exact sequence $$0 \to L \to R \to A \to 0 \leqno{(b)}$$ we have $\operatorname{depth} R \geq d$ if $s =d$ and $\operatorname{depth} R = s$ if $s < d$. Hence, thanks to the exact sequence
$$0 \to IR \to R \to G \to 0, \leqno{(c)}$$ we get $\operatorname {depth} G \geq d-1$ if $s = d$. If $s < d$, then $\operatorname{depth} R = s$, so that by \cite[Theorem 2.1]{HM1} we get $\operatorname{depth} G = s-1$.

Suppose that  $\operatorname {depth} G \geq d-1$. Then $\operatorname {depth} R \geq d$ by \cite[Theorem 2.1]{HM1}, whence by the exact sequence $(b)$ we have $\operatorname{depth}_TL \geq d$, so that $\operatorname{depth}_TS \geq d$ by the exact sequence $(a)$. Hence $S$ is a Cohen-Macaulay $T$-module. 
\end{proof}

Combining Lemma \ref{lemma1} (3) and Proposition \ref{proposition2}, we get the following result of D. G. Northcott and C. Huneke.

\begin{cor}[\cite{H, No}]\label{huneke}
We have $\e_1 \geq \e_0 - \ell_A(A/I)$. The equality $\e_1 = \e_0 - \ell_A(A/I)$ holds true if and only if $I^2 = QI$.
When this is the case, $\e_i = 0$ for all $2 \leq i \leq d$, provided $d \geq 2$.
\end{cor}

The following result is the heart of this paper.

\begin{thm}\label{MainTheoremA}
The following conditions are equivalent.
\begin{itemize}
\item[$(1)$] $\fkm S=(0)$ and ${\rank}_BS=1$.
\item[$(2)$] Either $S \cong B(-1)$ as graded $T$-modules, or $S \cong \a$ as graded $T$-modules for 
some graded ideal $\a~(\ne B)$ of $B$ with ${\height}_B \a \geq 2$.
\end{itemize}
\end{thm}

\begin{proof}
We have only to show $(1) \Rightarrow (2)$.
Because $S_1 \ne (0)$ and $S =\sum_{n \geq 1}S_n$ by Lemma \ref{lemma1}, we have $S \cong B(-1)$ as graded $B$-modules once $S$ is $B$-free.

Suppose that $S$ is not $B$-free. The $B$-module $S$ is torsionfree, since ${\Ass}_TS=\{\m T\}$ by Proposition  \ref{proposition2} (1).  Therefore, since $\operatorname{rank}_BS = 1$, we see $d \geq 2$ and $S \cong \a(m)$ as graded $B$-modules for some integer $m$ and  some graded ideal $\a~(\ne B)$ in $B$, so that we get the exact sequence 
$$0 \to S(-m) \to B \to B/\a \to 0$$
of graded $B$-modules.  We may assume that $\operatorname{ht}_B \a \geq 2$, since $B = k[X_1, X_2, \cdots, X_d]$ is the polynomial ring over the field $k = A/\m$. We then have $m \geq 0$, since $\a_{m+1} =[\a (m)]_1 \cong S_1 \ne (0)$ and $\a_0 = (0)$. We want to show $m = 0$.

Because $\dim B/\a \leq d-2$, the Hilbert polynomial of $B/\a$ has degree at most $d-3$. Hence 
\begin{eqnarray*}
\ell_A(S_n)&=&\ell_A(B_{m+n})-\ell_A([B/\a]_{m+n})\\
            &=&\binom{m+n+d-1}{d-1}-\ell_A([B/\a]_{m+n})\\
 &=&\binom{n+d-1}{d-1} +m\binom{n+d-2}{d-2} +\mbox{(lower terms)}
\end{eqnarray*}
for $n \gg 0$. Consequently

\begin{eqnarray*}
\ell_A(A/I^{n+1})&=&\e_0\binom{n+d}{d} -(\e_0 - \ell_A(A/I)){\cdot}\binom{n+d-1}{d-1}-\ell_A(S_n)\\
            &=&\e_0\binom{n+d}{d} -(\e_0 - \ell_A(A/I)+1){\cdot}\binom{n+d-1}{d-1}-m\binom{n+d-2}{d-2}\\
& & + \mbox{(lower terms)}
\end{eqnarray*}
by Proposition \ref{proposition2} (2), so that we get $\e_2 = -m$. Thus $m = 0$, because $\e_2 \geq 0$ by Narita's theorem (\cite{Na}). 
\end{proof}

We note some consequences of Theorem \ref{MainTheoremA}.

\begin{cor}\label{Theorem2}
Suppose $\e_1 = \e_0 - \ell_A(A/I) + 1$ and $I^3 =QI^2$. Let $c = \ell_A(I^2/QI)$. Then the following assertions hold true.
\begin{enumerate}
\item[$(1)$] $0 < c \leq d$ and $\mu_B(S) = c$.
\item[$(2)$] $\operatorname{depth} G \geq d-c$ and $\operatorname{depth}_BS = d-c + 1$. 
\item[$(3)$] $\operatorname{depth} G =d-c$, if $c \geq 2$.
\item[$(4)$] Suppose $c < d$. Then 
$\ell_A(A/I^{n+1})=\e_0\binom{n+d}{d} -\e_1\binom{n+d-1}{d-1}+\binom{n+d-c-1}{d-c-1}$ for all $n \geq 0$ and 
\[\e_i =\left\{
\begin{array}{
rl}
0 \ \ \ \ \ & \quad \mbox{if $i \ne c+1$} \\
(-1)^{c+1} & \quad \mbox{if $i = c+1$}
\end{array}\right.\]
for $2 \leq i \leq d$.
\item[$(5)$] Suppose $c=d$. Then 
$\ell_A(A/I^{n+1})=\e_0\binom{n+d}{d} -\e_1\binom{n+d-1}{d-1}$ for all $n \geq 1$. We have
$\e_i = 0$ for all $2 \leq i \leq d$, if $d \geq 2$. 
\end{enumerate}
\end{cor}

\begin{proof}

We have $\m S = (0)$ and $\operatorname{rank}_BS = 1$ by Proposition \ref{proposition2} (3), while $S = TS_1$ since $I^3=QI^2$ (cf. Lemma \ref{lemma1} (5)). Therefore  by Theorem  \ref{MainTheoremA}
we have $S \cong \a$ as graded $B$-modules where $\a = (X_1, X_2, \cdots, X_c)$ is  an ideal in $B$ generated by linear forms $\{X_i\}_{1 \leq i \leq c}$. Hence $0 < c \leq d$, $\mu_B(S) = c$,  and $\operatorname{depth}_BS = d - c + 1$, so that assertions (1), (2), and (3) follow (cf. Proposition \ref{proposition2} (4)).  Considering the exact sequence
$$0 \to S \to B \to B/\a \to 0$$
of graded $B$-modules, we get
\begin{eqnarray*}
\ell_A(S_n)&=&\ell_A(B_n) -\ell_A([B/\a]_{n})\\
            &=&\binom{n+d-1}{d-1} -\binom{n+d-c -1}{d-c-1}\\
\end{eqnarray*}
for all $n \geq 0$ (resp. $n \geq 1$), if $c < d$ (resp. $c = d)$. Thus assertions (4) and (5) follow (cf. Proposition \ref{proposition2} (2)). 
\end{proof}

Let $\tilde{I} = \bigcup_{n \geq 1}[I^{n+1} : I^n]$ be the Ratliff-Rush closure of $I$ (\cite{RR}), which is the largest $\m$-primary ideal in $A$ such that $I \subseteq \tilde{I}$ and $\e_i (\tilde{I}) = \e_i$ for all $0 \leq i \leq d$.

\begin{cor}\label{corollary1}
Suppose that $d \geq 2$. Then the following three conditions are equivalent to each other.
\begin{enumerate}
\item[$(1)$] $S \cong B_+$ as graded $T$-modules.
\item[$(2)$] $\e_1 = \e_0 - \ell_A(A/I) + 1$, $I^3 = QI^2$, and $\e_i = 0$ for all $2 \leq i \leq d$. 
\item[$(3)$] $I^3 = QI^2$, $\ell_A(\tilde{I}/I) =1$, and $\tilde{I}^2 ~= Q\tilde{I}$. 
\end{enumerate}
When this is the case, $\operatorname{depth} G = 0$. 
\end{cor}

\begin{proof}
Let $c = \ell_A(I^2/QI)$.

$(1) \Rightarrow (2)~and~the~last~assertion$  This follows from Corollary  \ref{Theorem2}. Notice that $c= \ell_A(S_1) = d$ and $I^3 = QI^2$, because $S \cong B_+$.

$(2) \Rightarrow (1)$ We have $c = d$ by Corollary \ref{Theorem2} (4), (5),  because $\e_i = 0$ for all $2 \leq i \leq d$, so that $S \cong B_+$ (see Proof of Corollary \ref{Theorem2}).

$(2) \Rightarrow (3)$ We have $\operatorname{depth} G = 0$ by Corollary \ref{Theorem2} (3), since $c=d$. Now we apply local cohomology functors ${\H}_M^i(*)$ of $T$ with respect to the graded maximal ideal $M = \m T + T_+$ to the exact sequences
$$0 \to IR \to R \to G \to 0~~~\operatorname{and}~~~0 \to IT \to IR \to S \to 0$$
of graded $T$-modules and we have the monomorphism $${\H}_M^0(G) \hookrightarrow {\H}_M^1(IR)$$ and the isomorphisms $${\H}_M^1(IR) \cong {\H}_M^1(S) \cong B/B_+$$ of graded $T$-modules (recall that $S \cong B_+$ and $IT$ is a Cohen-Macaulay $T$-module with $\dim_TIT = d+1$). Consequently, because  ${\H}_M^0(G) \ne (0)$ and $\ell_A(B/B_+)=1$, we get
$${\H}_M^0(G) \cong {\H}_M^1(IR) \cong {\H}_M^1(S) \cong B/B_+,$$
whence ${\H}_M^0(G) = [{\H}_M^0(G)]_0 \ne (0)$. Thus $\ell_A(\tilde{I}/I) = 1$ since $[{\H}_M^0(G)]_0 \cong \tilde{I}/I$. Therefore it follows from the equality $\e_1 = \e_0-\ell_A(A/I) + 1$ that 
$$\e_1(\tilde{I})= \e_0(\tilde{I}) - \ell_A(A/\tilde{I}),$$
because $\e_i (\tilde{I}) = \e_i$ for $i =0,1$ and $\ell_A(A/I) = \ell_A(A/\tilde{I}) + 1$. Hence $\tilde{I}^2 = Q\tilde{I}$ by Corollary \ref{huneke}.

$(3) \Rightarrow (2)$ We have $\e_1 = \e_0 - \ell_A(A/I) + 1$ and $\e_i = 0$ for all $2 \leq i \leq d$, since $\e_1(\tilde{I})= \e_0(\tilde{I}) - \ell_A(A/\tilde{I}) = \e_0 - \ell_A(A/I ) + 1$ and $\e_i(\tilde{I}) = 0$ for all $2 \leq i \leq d$ (cf. Corollary \ref{huneke}).
\end{proof}

Let us include a proof of Theorem \ref{Sally} in our context, in order to show how our arguments work. 

\begin{proof}[Proof of Theorem \ref{Sally}]

$(1) \Rightarrow (3)$ See Lemma \ref{lemma1} (2), (5). 

$(3) \Rightarrow (1)$ By Lemma \ref{lemma1} (5) we have $S = TS_1$, whence $\m S = (0)$ because $S_1 \cong I^2/QI$ and $\ell_A(I^2/QI)=1$. Therefore we have an epimorphism $B(-1) \to S \to 0$ which has to be an  isomorphism, since $\dim_TS=d$. 

$(1) \Rightarrow (2)~and~the~last~assertions$  We have $I^3=QI^2$ since $S = TS_1$, whence the assertions follows from Corollary \ref{Theorem2} (notice that $c = 1$).

$(2) \Rightarrow (1)$ We have $\m S = (0)$ and $\operatorname{rank}_BS=1$ by Proposition \ref{proposition2} (3), while the $B$-module $S$ is torsionfree by Proposition \ref{proposition2} (1). Hence $S$ is $B$-free if $d=1$, so that $S \cong B(-1)$ as graded $T$-modules (notice that $S_1 \ne (0)$).

Assume that $d=2$. Then we have  an exact sequence
$$0 \to B(-1) \to S \to C \to 0 \leqno{(a)}$$
of graded $B$-modules with $\dim_BC \leq 1$. Therefore $\ell_A(S_n) = \ell_A(B_{n-1}) + \ell_A(C_n) = 
\binom{n}{1} + \ell_A(C_n)$ for all $n \geq 1$, so that by Proposition \ref{proposition2} (2)
$$\ell_A(A/I^{n+1}) = \e_0 \binom{n+2}{2} - (\e_0 - \ell_A(A/I) + 1)\binom{n+1}{1} + (1 - \ell_A(C_n)).$$ Consequently $\e_2 = 1 - \ell_A(C_n) > 0$ by Narita's theorem \cite{Na}. Hence $\ell_A(C_n)  = 0$ for all $n \geq 1$. Thus $\ell_A(C) \leq 1$, so that $C = (0)$ by exact sequence (a).

Now let $ d \geq 3$ and assume that our assertion holds true for $d - 1$. Choose the element $a_1 \in Q$ so that $a_1$ is a superficial element of $I$ (this choice is possible, because the field $A/\m$ is infinite). Let $\overline{A} = A/(a_1)$, $\overline{I} = I/(a_1)$, and $\overline{Q} = Q/(a_1)$. Then all the assumptions of condition (2) are safely fulfilled for the ideal $\overline{I}$ in $\overline{A}$, since $\e_i(\overline{A}) = \e_i$ for all $0 \leq i \leq d-1$. Consequently the hypothesis of induction yields that $\operatorname {depth} \mathrm{G}(\overline{I}) \geq (d-1)-1 = d - 2 > 0$ and so, thanks to Sally's technique \cite{S3}, we see that $a_1t$ is a nonzerodivisor for $G$, whence $I^3 = QI^2$ because $\overline{I}^3 =\overline{Q}~\overline{I}^2$. Thus $S \cong B(-1)$ as graded $B$-modules by Corollary \ref{Theorem2} (notice that $c = 1$). 
\end{proof}

\section{Proof of Theorem \ref{MainTheorem}}
The purpose of this section is to prove Theorem \ref{MainTheorem}. Let us begin with the following.

\begin{thm}\label{d=2}
Suppose that $d = 2$. Then the following three conditions are equivalent to each other.
\begin{enumerate}
\item[$(1)$] $\e_1= \e_0 - \ell_A(A/I) + 1$.
\item[$(2)$] Either $S \cong B(-1)$ as graded $T$-modules or $S \cong B_+$ as graded $T$-modules.
\item[$(3)$] Either $(a)$ $I^3 = QI^2$ and $\ell_A(I^2/QI) = 1$, or $(b)$ $\ell_A(\tilde{I}/I) = 1$ and $\tilde{I}^2 = Q\tilde{I}$.
\end{enumerate}
We get $\e_2 = 1$ $(resp.~\e_2=0)$ if condition $(3)~(a)$ $(resp.~condition~(3)~ (b))$ is satisfied.
 and furthermore have the following
$$\begin{array}{c|c|c|cc}
\mathrm{e}_2 & \mathrm{r}_Q(I) & \depth_B$S$ & \depth $G$ &  \\
\hline 
1 & 2 & 2 & 2 & if~Q\nsupseteq I^2 \\
\hline
1 & 2 & 2 & 1 & if~Q\supseteq I^2 \\
\hline
0 & 2 & 1 & 0 & $G$\ \text{is a Buchsbaum ring with}\ \Bbb{I}($G$)=2.
\end{array}$$

\end{thm}

\begin{proof}
$(1) \Rightarrow (2)$ Thanks to Corollary \ref{Theorem2} and its proof, we have only to show that $I^3 =QI^2$. This equality directly follows from a result of M. Rossi \cite[Corollary 1.5]{R}. Let us note a proof in our context for the sake of completeness.

We have $\m S = (0)$ and $\operatorname{rank}_BS = 1$. Assume that $S \not\cong B(-1)$ as graded $B$-modules. Then by Theorem \ref{MainTheoremA} we have $S \cong \a$ as graded $B$-modules for some graded ideal $\a \ne B$ with $\operatorname{ht}_B\a =2$. We will show that $\a = B_+$. Since $\a_1 \cong S_1 \neq (0)$, the ideal $\a$ contains a linear form $f \ne 0$ of $B$, so that the ideal $\a/(f)$ of $B/(f)$ is principal, since $B/(f)$ is the polynomial ring  with one indeterminate over the field $k=A/\fkm$. We write $\a=(f,g)$ with a  form $g \in B$. Then $f, g$ is a regular sequence in $B$, since ${\height}_B\a = 2$. Let $\alpha={\deg}~g$. Then $\alpha \leq 2$ by Lemma \ref{lemma1} (4). We will show that $\alpha=1$. 

Assume that $\alpha =2$. Then, since $S \cong \a =(f,g)$, the graded $B$-module $S$ has a resolution of the form
$$0 \to B(-3) \overset{\binom{g}{f}}{\to} B(-1) \oplus B(-2) \overset{\varphi = 
\begin{pmatrix}
\xi & \eta
\end{pmatrix}
}{\longrightarrow} S \to 0,$$
in which the homomorphism $\varphi$ is defined by $\varphi({\rm{\bf{e}}}_1) = \xi \in S_1$ and $\varphi({\rm{\bf{e}}}_2) = \eta \in S_2$ (here $\{\rm{\bf{e}}_1, \rm{\bf{e}}_2\}$ denotes the standard basis of $B(-1) \oplus B(-2)$). Let $a \in Q$, $c \in Q^2$, $x \in I^2$, and $y \in I^3$ such that
$f$  and $g$ are, respectively, the images of $at$ and $ct^2$ in $B$ and
$\xi$ and $\eta$ are, respectively, the images of $xt$ and $yt^2$ in $S$. We notice that $a \notin \fkm Q$ so that $Q=(a,b)$ for some $b \in Q$. Hence $c=a^2 z_1+ab z_2+b^2 z_3$
for some $z_1$, $z_2$, and $z_3 \in A$.

Let us now consider the relation $g \xi+ f \eta=0$ in $S_3$, that is, $cx+ay \in Q^3I$.
We write $cx+ay=(a^2 z_1+ab z_2+b^2 z_3)x + ay = a^2i+b^2j$ with $i$, $j \in QI$ (recall that $Q^3=(a^2,b^2)Q)$. We then have that $ay'=b^2x'$, where $y'=y+az_1x+bz_2x-ai$ and $x'=j-z_3x$. Therefore $x'=ah$ and $y'=b^2h$ for some $h \in A$, because the sequence $a$, $b^2$ is $A$-regular. Hence $h \in I^3 : (a^2,b^2) \subseteq \widetilde{I}$, because $a^2h = ax' \in I^3$ and $b^2h = y' \in I^3$. Now notice that $S= B\xi + B\eta$. We then have $S_1=B_0 \xi$ and $S_2=B_1S_1+B_0\eta$, whence $\ell_A(I^2/QI)=1$ and $I^3=QI^2+(y)$.

We need the following.

\begin{claim}\label{claim1}
$h \not\in I$ and $x' = ah \not\in QI$.
\end{claim}

\begin{proof}
Assume that $h \in I$.
Then $y'=b^2h \in Q^2I$  so that $y=y'-az_1x-bz_2x+ai \in QI^2$, whence $I^3=QI^2+(y)=QI^2$. This forces $S=BS_1$, which is impossible because $\alpha = 2$. 
Thus $h \notin I$. Suppose $ah \in QI$ and let $ah = ai_1 +bi_2$ with $i_1, i_2 \in I$. Then $a(h -i_1) = bi_2$ and so $h -i_1 \in (b)$. Hence $h \in I$, which is impossible.
\end{proof}

Because $\ell_A(\tilde{I}/I) \geq 1$ by this claim, we get the following.
\begin{eqnarray*}
\e_1&=&\e_0-\ell_A(A/I)+1 \\
&=&\e_0(\widetilde{I})-\ell_A(A/\widetilde{I})-(\ell_A(\widetilde{I}/I)-1)\\
&\leq& \e_0(\widetilde{I})-\ell_A(A/\widetilde{I})\\
&\leq& \e_1(\widetilde{I}) \\
&=&\e_1,
\end{eqnarray*}
where $\e_0(\tilde{I}) - \ell_A(A/\widetilde{I}) \leq \e_1(\widetilde{I}) $ is the inequality of Northcott for the ideal $\tilde{I}$ (cf. Corollary \ref{huneke}).  Then we have $\ell_A(\widetilde{I}/I)=1$ and
$\e_1(\widetilde{I})=\e_0(\widetilde{I})-\ell_A(A/\widetilde{I})$, so that 
$\widetilde{I}=I+(h)$ and $\widetilde{I}^2=Q\widetilde{I}$ by Corollary \ref{huneke}, since $Q$ is also a reduction of $\widetilde{I}$. Thus the associated graded ring of $\widetilde{I}$ is a Cohen-Macaulay ring and so
$(a) \cap \widetilde{I}^n=a\widetilde{I}^{n-1}$ for all $n \in \Z$, because $at$ is $\mathrm{G}(\tilde{I})$-regular.

Now recall that $x'= ah  \notin QI$ and we have $I^2=QI+(ah)$, since $\ell_A(I^2/QI) =1$. Let $\overline{A}=A/(a)$, $\overline{I}=I/(a)$, and $\overline{Q}=Q/(a)$.
Then $\overline{I}^2=\overline{Q}\,\overline{I}$,
and so $\overline{I}^3=\overline{Q}~{\overline{I}}^2$, whence $I^3 \subseteq QI^2+(a)$. Thus $I^3 = QI^2 + [(a) \cap I^3]$. On the other hand 
$$(a) \cap I^3 \subseteq (a) \cap \widetilde{I}^3=a\widetilde{I}^2=aQ\widetilde{I}
=(aQ)(I+(h))=(aQ)I+x'Q \subseteq QI^2,$$
whence $I^3=QI^2$ so that $\alpha = 1$, which is the required contradiction. Thus $S = BS_1$ and  $S \cong B_+$.

$(2) \Rightarrow (3)$ See Theorem \ref{Sally} and Corollary \ref{corollary1}.

$(3) \Rightarrow (1)$  If condition (a) is satisfied, we have by Theorem \ref{Sally} assertion (1). Suppose condition (b) is satisfied. Then $\e_1 = \e_1(\tilde{I}) = \e_0(\tilde{I}) - \ell_A(A/\tilde{I}) = \e_0 - \ell_A(A/I) + 1$ (cf. Corollary \ref{huneke}). 

We now consider the last assertions. Suppose condition (3) (a) is satisfied. Then $\e_2=1$ by Theorem \ref{Sally}. If $Q \supseteq I^2$, then $I^2 =Q \cap I^2 \ne QI$, so that $G$ is not a Cohen-Macaulay ring. If $Q \not\supseteq I^2$, then $Q \cap I^2 = QI$ because $\ell_A(I^2/QI) = 1$ and $I^2 \supsetneq Q \cap I^2 \supseteq QI$. Since $I^3 = QI^2$, this yields $G$ is a Cohen-Macaulay ring.

Suppose condition (3) (b) is satisfied. Then, since  $\tilde{I}^2 = Q\tilde{I}$, $\e_2=0$ by Corollary \ref{huneke}  (recall that $\e_2(\tilde{I}) = \e_2$) and $\mathrm{R}'(\tilde{I})$ is a Cohen-Macaulay ring.  We furthermore have the following.

\begin{claim}\label{claim2}
$\tilde{I}^n = I^n$ for all $n \geq 2$.
\end{claim}

\begin{proof}
We have $S \cong B_+$ as graded $T$-modules, because $\e_2=0$. Hence ${\H}_M^0(G) = [{\H}_M^0(G)]_0$, thanks to Proof of Corollary \ref{corollary1}, $(2) \Rightarrow (3)$. Let $n \geq 2$ be an integer. We then have 
$$[\widetilde{I^n} \cap I^{n-1}]/I^n \cong [{\H}_M^0(G)]_{n-1} = (0).$$
Consequently $\tilde{I}^n = I^n$, because $\tilde{I}^n \subseteq \widetilde{I^n} \cap I^{n-1}$ (recall that $\tilde{I}^n = Q^{n-1}\tilde{I}$, since $\tilde{I}^2 =Q\tilde{I}$). Thus $\tilde{I}^n = I^n$ for all $n \geq 2$.
\end{proof}

We put $W = \mathrm{R}'(\tilde{I})/R'$ and look at the exact sequence
$$
0 \to R' \to \mathrm{R}'(\tilde{I}) \to \mathrm{R}'(\tilde{I})/R' \to 0 \leqno{(\sharp{})}
$$
of graded $R'$-modules. Notice that $W = W_1 \cong \tilde{I}/I$ by Claim \ref{claim2} whence $\ell_A(W) = 1$. Let $N = (\m , R_+, t^{-1})R'$ be the unique graded maximal ideal in $R'$. Then because $\mathrm{R}'(\tilde{I})$ is a Cohen-Macaulay ring,  applying functors ${\H}_N^i(*)$ to the exact sequence $(\sharp{})$, we see that ${\H}_N^i(R') = (0)$ for all $i \ne 1, 3$ and ${\H}_N^1(R') = W$. Thus $R'$ is a Buchsbaum ring with the Buchsbaum invariant 
$$\Bbb{I} (R') = \sum_{i=0}^2\binom{2}{i}\ell_A({\H}_N^i(R')) = 2,$$  whence so is the graded ring $G=R'/t^{-1}R'$. We similarly have that $R$ is a Buchsbaum ring with $\Bbb{I} (R) = 2$, because $\mathrm{R}(\tilde{I})$ is a Cohen-Macaulay ring and $\mathrm{R}(\tilde{I})/R =[\mathrm{R}(\tilde{I})/R]_0 \cong \tilde{I}/I$. This completes the proof of Theorem \ref{d=2}.
\end{proof}

We are now in a position to prove Theorem \ref{MainTheorem}.

\begin{proof}[Proof of Theorem \ref{MainTheorem}]
$(1) \Rightarrow  (3)$ 
We have $\e_1 = \e_0 - \ell_A(A/I) + 1$ by Proposition \ref{proposition2} (3) and so $\e_2 = 0$ by Theorem \ref{Sally}. Because $S \not\cong B(-1)$, by Theorem  \ref{MainTheoremA} we get $S \cong \a$ as graded $B$-modules for some graded ideal $\a~(\ne B)$ in $B$ with ${\height}_B\a \geq 2$. Since $\mu_B(\a) = \mu_B(S) =2$, the ideal $\a$ is a complete intersection with $\mathrm{ht}_B\a = 2$, so that $\operatorname{depth}_BB/\a= d-2$, whence $\operatorname{depth}_BS = d-1$. Thus $\operatorname{depth} G = d - 2$ by Proposition \ref{proposition2} (4).

$(3) \Rightarrow  (2)$
First of all let us show that $I^3 = QI^2$. 
Thanks to Theorem \ref{d=2}, we may assume that $d \geq 3$ and our assertion holds true for $d-1$. Since $\operatorname{depth} G \geq d-2 >0$, we may choose $a_1 \in Q$ so that $a_1t$ is a nonzerodivisor in $G$. Let $\overline{A} = A/(a_1), \overline{I} = I/(a_1),$ and $\overline{Q}  = Q/(a_1)$. Then, because $\mathrm{G}(\overline{I}) \cong G/a_1t{\cdot}G$ and $\e_i(\overline{I})=\e_i$ for all $0 \leq i \leq d-1$, we see condition (3) is satisfied for the ideal $\overline{I}$, so that $\overline{I}^3 = \overline{Q}~\overline{I}^2$ whence $I^3 = QI^2$.  Therefore, since $\e_2=0$, we see in Corollary  \ref{Theorem2} that $c = \mu_B(S) = 2$, whence assertion (2) follows (cf. Proof of Corollary \ref{Theorem2}).

$(2) \Rightarrow  (4)$
We have $\m S =(0)$, $S = TS_1$, and $S_1 \cong B_0^2$. Hence $\m I^2 \subseteq QI$, $I^3 = QI^2$, and $\ell_A(I^2/QI) = \ell_A(S_1) = 2$. We similarly have $$\ell_A(I^3/Q^2I) = \ell_A(S_2) = 2\ell_A(B_1) - \ell_A(B_0) = 2d -1 < 2d.$$

$(4) \Rightarrow  (1)$
We have $S=TS_1$ and so $\m S = (0)$, since $\m S_1 =(0)$. Because $\ell_A(S_1) = 2$, we have an epimorphism $B(-1)^2 \to S \to 0$ of graded $B$-modules, which cannot be an isomorphism since $\ell_A(S_2) = \ell_A(I^3/IQ^2) < 2d$. Thus $\operatorname{rank}_BS = 1$, so that we have $\mu_B(S) = 2$ by Corollary  \ref{Theorem2}.

See Theorem \ref{d=2} for the equivalence between condition (5) and the others.
See Corollary  \ref{Theorem2} and Proof of Theorem \ref{d=2} for the last assertions.
\end{proof}

We note the following.

\begin{ex}\label{ex1}
Let $A=k[[ X, Y, Z_1, Z_2, \cdots, Z_m ]]$~$(m \geq 0)$
be the formal power series ring over a field $k$. Hence $\dim A = m+2$. 
We put
$$ Q=(X^4,Y^4,Z_1,Z_2,\cdots,Z_m) \mbox{  and  } I=Q+(X^3Y,XY^3).$$
Then
$$\mbox{$\fkm I^2 \subseteq QI$, $\ell_A(I^2/QI)=2$, $\ell_A(I^3/Q^2I)<2d$, and $I^3=QI^2,$}$$
where $d = m + 2$. 
Hence condition (4) in Theorem \ref{MainTheorem} is satisfied, so that $\m S = (0),~\operatorname{rank}_BS = 1$, and $\mu_B(S) = 2$.
We have $\mbox{$\ell_A(A/Q)=16$ and $\ell_A(A/I)=11$}$
and 
$$\ell_A(A/I^{n+1})=16\binom{n+2}{2}-6\binom{n+1}{1}$$
for all $n \geq 1$, if $m=0$.
If $m \geq 1$, we get
$$\ell_A(A/I^{n+1})=16\binom{n+d}{d}-6\binom{n+d-1}{d-1}+\binom{n+d-3}{d-3} $$
for all $n \geq 0$, whence $\e_3=-1$ and $\e_i = 0$ ($2 \leq i \leq d, i \ne 3$).
\end{ex}

\begin{proof}
Because $G = \mathrm{G}((X^4, X^3Y, XY^3, Y^4))[Z_1, Z_2, \cdots, Z_m]$~(the polynomial ring), the case where $m >0$ follows easily from the case $m = 0$ (see Theorem \ref{MainTheorem} (3)).  Let $m = 0$. Then $I^2=QI+(X^6Y^2, X^2Y^6)$. It is routine to show that $\fkm I^2 \subseteq QI$,
$\ell_A(I^2/QI)=2$, and $I^3=QI^2$.
We have
$QI^2=Q^2I+(X^{10}Y^2,X^6Y^6,X^2Y^{10})$,
whence $\ell_A(I^3/Q^2I) = 3$. 
\end{proof}

Before closing this section, let us study ideals with $\e_1 =2$.

\begin{thm}\label{e2}
Suppose that $\e_1=2$ and $I^2 \neq QI$.
Then the following assertions hold true.
\begin{enumerate}
\item[$(i)$] $\ell_A(I/Q)=\ell_A(I^2/QI)=1$.
\item[$(ii)$] $I^3=QI^2$.
\item[$(iii)$] $S \cong B(-1)$ as graded $T$-modules.
\item[$(iv)$] $\operatorname{depth} G=d-1$.
\item[$(v)$] $\e_2=1$, if $d \geq 2$ and $\e_i = 0$ for $3 \leq i \leq d$, if $d \geq 3$.
\end{enumerate}
\end{thm}

\begin{proof}
Since $I^2 \ne QI$, we get 
$$0 < \ell_A(I/Q) = \e_0 - \ell_A(A/I) < \e_1 =2$$
by Corollary \ref{huneke}. Therefore $\ell_A(I/Q) = 1$ and $\e_1 = \e_0  - \ell_A(A/I) + 1$. Let $I = Q + (x)$ with $x \in A$. Then $I^2 = QI + (x^2)$, so that $\ell_A(I^2/QI) = 1$ because $I^2 \ne QI$ and $\m I \subseteq Q$. We will show by induction on $d$ that $I^3 = QI^2$  and $\operatorname{depth} G \geq d-1$. Since $\ell_A(S_1) = \ell_A(I^2/QI) = 1$, thanks to Theorems \ref{Sally} and  \ref{d=2}, we may assume that $d \geq 3$ and our assertion holds true for $d - 1$. Choose $a_1 \in Q$ so that $a_1$ is a superficial element of $I$. Then, passing to the ideals $\overline{I} = I/(a_1)$ and $\overline{Q} = Q/(a_1)$ in the ring $\overline{A} = A/(a_1)$, we get $\e_1(\overline{I}) = \e_1 =2$. We claim that $\overline{I}^2 \ne \overline{Q}\,\overline{I}$. In fact, if $\overline{I}^2 = \overline{Q}\,\overline{I}$, then the ring $\mathrm{G}(\overline{I})$ is Cohen-Macaulay.  Hence Sally's technique \cite{S3}  works to get that $a_1t$ is regular on $G$, so that $I^2 =QI$, which is impossible. Consequently, the hypothesis of induction shows $\overline{I}^3 = \overline{Q}\,\overline{I}^2$ and $\operatorname{depth} \mathrm{G}(\overline{I}) \geq (d-1) -1 = d - 2 > 0$. Thus, thanks to Sally's technique again, we get $a_1t$ is regular on $G$, so that $I^3 = QI^2$ and $\operatorname{depth} G \geq d-1$. Since $\m I \subseteq Q$, we  get $I^2 \subseteq Q$, so that $G$ is not a Cohen-Macaulay ring; otherwise, $I^2 = Q \cap I^2 = QI$. Hence ${\depth}~G = d-1$. See Theorem \ref{Sally} for assertions (iii) and (v). 
\end{proof}

\begin{cor}
Suppose that $\e_1=2$. Then $\operatorname{depth} G \geq d-1$. The ring $G$ is Cohen-Macaulay if and only if $I^2 = QI$. 
\end{cor}


\section{Buchsbaumness in the graded rings $G$ associated to ideals with $\e_1=2$}

The purpose of this section is to study the problem of when the associated graded rings $G$ are  Buchsbaum for the ideals $I$ with $\e_1 = 2$.

We assume that $\e_1 =2$ but $I^2 \ne QI$. We have ${\depth}\,R=d$ (\cite[Theorem 2.1]{HM1}), because ${\depth}\,G=d-1$ by Theorem \ref{e2}. Let  $N = \fkm R+R_+$ and let
$$ {\rm{a}}_i(G)=\sup\{\mbox{$n \in \Bbb{Z}$ $|$ $[{\rmH}_N^i(G)]_n \neq (0)$}\}$$
for $0 \leq i \leq d$.

\begin{lem}\label{buch}
The following assertions hold true.
\begin{itemize}
\item[$(1)$] ${\rm{a}}_d(G)=2-d$ and $\ell_A([{\rmH}^d_N(G)]_{2-d})=1$.
\item[$(2)$] ${\rm{a}}_{d-1}(G)=1-d$ and $\ell_A([{\rmH}^{d-1}_N(G)]_{1-d})=1$.
\end{itemize}
In particular, ${\rmH}^{0}_N(G)=[{\rmH}^{0}_N(G)]_{0}$ and $G$ is a Buchsbaum ring, if $d = 1$. 
\end{lem}

\begin{proof}
Suppose $d=1$. Let  $a = a_1$ and $f= at$. Then $I^3=aI^2$ by Theorem \ref{e2}.
Let $n \geq 1$ be an integer and $x \in I^n$. Then  since $I^{n+2}=aI^{n+1}$, we get $x \in I^{n+1}$ if $ax \in I^{n+2}$. Thus $(0):_G f = [(0) :_G f]_0$. Hence $(0):_Gf^n = (0):_G f$ for all $n \geq 1$, so that $${\H}_N^0(G) =(0):_G f = [(0):_Gf]_0 \cong \tilde{I}/I.$$ In particular $\ell_A(\tilde{I}/I) >0$. Because
\begin{eqnarray*}
\e_1&=&\e_0-\ell_A(A/I)+1 \\
    &=&\e_0(\widetilde{I})-\ell_A(A/\widetilde{I})-(\ell_A(\widetilde{I}/I)-1)\\
    &\leq& \e_0(\widetilde{I})-\ell_A(A/\widetilde{I})\\
    &\leq& \e_1(\widetilde{I}) \\
    &=& \e_1,
\end{eqnarray*}
we get $\ell_A(\widetilde{I}/I)=1$, which proves assertion (2). In particular, ${\rmH}^{0}_N(G)=[{\rmH}^{0}_N(G)]_{0}$ and $G$ is a Buchsbaum ring. Because $(0):_G f={\H}_N^0(G)$, we have the following exact sequence
$$ 0 \to {\rmH}_N^0(G) \to G/fG \to {\rmH}_N^1(G)(-1) \overset{f}{\to} {\rmH}_N^1(G) \to 0 $$
of local cohomology modules. Hence ${\rma}_1(G) = 1$, because ${\H}_N^0(G) = [{\H}_N^0(G)]_0$ and $G/f G =A/I \oplus I/Q \oplus I^2/QI$ with $I^2/QI \ne (0)$. We have $[G/fG]_2 \cong [{\H}_N^1(G)]_1$, whence $\ell_A([{\H}_N^1(G)]_1) = \ell_A(I^2/QI) = 1$ by Theorem \ref{e2}.

Now we consider the case where $d \geq 2$.
Because $\operatorname{depth} G=d-1>0$ by Theorem \ref{e2},
we may assume that $f=a_1t$ is regular on $G$.
We put $\overline{A}=A/(a_1)$, $\overline{I}=I/(a_1)$, and $\overline{Q}=Q/(a_1)$.
Then $\e_1(\overline{I})=2$ and $\overline{I}^2 \neq \overline{Q}\,\overline{I}$
(cf. Proof of Theorem \ref{e2}).
Hence,  thanks to the hypothesis of induction, we have assertions (1) and (2) for the ideal $\overline{I}$. We now look at the exact sequence
$$ 0 \to {\rmH}_N^{d-2}(\Gr(\overline{I})) \to {\rmH}_N^{d-1}(G)(-1) \overset{f}{\to} 
 {\rmH}_N^{d-1}(G) \to {\rmH}_N^{d-1}(\Gr(\overline{I})) \leqno{(*)}$$
$$\to {\rmH}_N^d(G)(-1) \overset{f}{\to} {\rmH}_N^d(G) \to 0 $$
of local cohomology modules which is induced from the canonical exact sequence 
$$0 \to G(-1) \overset{f}{\to} G \to \Gr(\overline{I}) \to 0$$
of graded $G$-modules.
Then, since ${\rma}_{d-2}(\Gr(\overline{I}))=2-d$, we get a monomorphism $[{\rmH}_N^{d-1}(G)]_n \hookrightarrow [{\rmH}_N^{d-1}(G)]_{n+1}$
for all $n \geq 2-d$, whence $[{\rmH}_N^{d-1}(G)]_n=(0)$ for all $n \geq 2-d$.
Thus  ${\rma}_{d-1}(G) \leq 1-d$ and
$$[{\rmH}_N^{d-2}(\Gr(\overline{I}))]_{2-d} \cong [{\rmH}_N^{d-1}(G)]_{1-d}.$$
Therefore ${\rma}_{d-1}(G) =1-d$ and $\ell_A([{\rmH}_N^{d-1}(G)]_{1-d})=\ell_A([{\rmH}_N^{d-2}(\Gr(\overline{I}))]_{2-d})=1$.
On the other hand, letting ${\rma}={\rma}_d(G)$, in exact sequence $(*)$ above we see that
$[{\rmH}_N^d(G)(-1)]_{{\rma} + 1}=[{\rmH}_N^d(G)]_{\rma}~(\neq (0))$
is a homomorphic image of $[{\rmH}_N^{d-1}(\Gr(\overline{I}))]_{\rma +1}$.
Hence ${\rma} +1 \leq {\rma}_{d-1}(\Gr(\overline{I}))= 3-d$, whence ${\rma} \leq 2-d$.
Because $[{\rmH}_N^{d-1}(G)]_{3-d}=(0)$ and $[{\rmH}_N^d(G)]_{3-d}=(0)$, by exact sequence $(*)$ 
we have $[{\rmH}_N^{d-1}(\Gr(\overline{I}))]_{3-d} \cong [{\rmH}_N^d(G)]_{2-d}$.
Consequently, ${\rma}_d(G)=2-d$ and $\ell_A([{\rmH}_N^d(G)]_{2-d})=1$, as is claimed.
\end{proof}

We are in a position to state the main result of this section. See Theorem \ref{ex2} for an example whose associated graded ring $G$ is  a Buchsbaum ring.

\begin{thm}\label{MainTheoremD}
The following two conditions are equivalent to each other.
\begin{itemize}
\item[$(1)$] $G$ is a Buchsbaum ring.
\item[$(2)$] ${\rmH}^{d-1}_N(G)=[{\rmH}^{d-1}_N(G)]_{1-d}$.

\vspace{2mm}

\noindent
\hspace{-13.5mm}\textup{When $d\geq2$, one can add the following.}
\item[$(3)$] $R$ is a Buchsbaum ring.
\end{itemize}
\end{thm}

\begin{proof}
$(2) \Rightarrow (1)$
By Lemma \ref{buch} we have $N {\cdot} {\rmH}_N^{d-1}(G)=0$, since $\m{\cdot}[{\rmH}_N^{d-1}(G)]_{1-d}=(0)$. 
Hence $G$ is a Buchsbaum ring, because ${\depth}\,G=d-1$ by Theorem \ref{e2}.

$(1) \Rightarrow (2)$
By Lemma \ref{buch} we may assume that $d \geq 2$ and our assertion holds true for $d-1$.
Since ${\depth}\,G=d-1>0$, we may assume that $f=a_1t$ is regular on $G$. Similarly as before, let $\overline{A}=A/(a_1)$, $\overline{I}=I/(a_1)$, and $\overline{Q}=Q/(a_1)$. Then $\Gr(\overline{I}) = G/fG$ is a Buchsbaum ring with ${\depth}\,\mathrm{G}(\overline{I})=d-2$. Hence, thanks to the hypothesis of induction, we get 
${\rmH}_N^{d-2}(\Gr(\overline{I}))=[{\rmH}_N^{d-2}(\Gr(\overline{I}))]_{2-d}$.
Thus ${\rmH}_N^{d-1}(G)=[{\rmH}_N^{d-1}(G)]_{1-d}$, because ${\rmH}_N^{d-2}(\Gr(\overline{I})) \cong {\rmH}_N^{d-1}(G)(-1)$ (see the exact sequence $(*)$ in Proof of Lemma \ref{buch}).

Suppose that $d \geq 2$.

$(3) \Rightarrow (1)$
Apply functors ${\H}_N^i(*)$ to 
the exact sequences
$$ 0 \to R_+ \to R \to A \to 0 \hspace{3.0mm}\mbox{ and }\hspace{3.0mm} 0 \to R_+(1) \to R \to G \to 0.$$ Then, since $\operatorname{depth} R = d$ (cf. \cite[Theorem 2.1]{HM1}),
we get the exact sequences
$$\vspace{-3.0mm}0 \rightarrow {\rmH}^d_N(R_+) \rightarrow {\rmH}^d_N(R) \rightarrow {\rmH}^d_{\fkm}(A) \mbox{  and } $$
\vspace{-3.0mm}
$(**)$
$$0 \rightarrow {\rmH}^{d-1}_N(G) \rightarrow {\rmH}^{d}_N(R_+)(1) \rightarrow {\rmH}^d_N(R) \rightarrow {\rmH}^d_N(G).$$
Because   $R$ is a Buchsbaum ring, $N {\cdot} {\rmH}_N^d(R)=(0)$
and so $N {\cdot} {\rmH}_N^{d}(R_+)=(0)$.
Thus $N {\cdot} {\rmH}_N^{d-1}(G)=(0)$, whence $G$ is a Buchsbaum ring.

$(2) \Rightarrow (3)$
Look at exact sequences $(**)$.
Then 
$$[{\rmH}_N^d(R_+)]_{n+1} \to \hspace{-5.0mm}\to [{\rmH}_N^d(R)]_n$$
for all $n > {\rma}_d(G)=2-d$.
Hence 
$$[{\rmH}_N^d(R)]_n\cong [{\rmH}_N^d(R_+)]_n=(0)$$
for all $n > 2-d$.
We have
$$[{\rmH}_N^d(R_+)]_n \cong [{\rmH}_N^d(R)]_n$$ for all $n < 0$ and
$$[{\rmH}_N^d(R_+)]_n = [{\rmH}_N^d(R_+)(1)]_{n-1} \hookrightarrow [{\rmH}_N^d(R)]_{n-1}$$
for all $n < 2-d$, since ${\rmH}_N^{d-1}(G)=[{\rmH}_N^{d-1}(G)]_{1-d}$. Therefore, since $d \geq 2$, $[{\rmH}_N^d(R)]_n$ is embedded into $[{\rmH}_N^d(R)]_{n-1}$ for all $n<2-d$.
Hence $[{\rmH}_N^d(R)]_n=(0)$ for all $n <2-d$, because ${\rmH}_N^d(R)$ is a finitely graded $R$-module (cf. \cite{BS}; recall that $G$ is a Buchsbaum ring).
Thus $${\rmH}_N^d(R)=[{\rmH}_N^d(R)]_{2-d}.$$

Because $[{\rmH}_N^d(R_+)]_{3-d}=(0)$, by exact sequence $(**)$ we have
$$ [{\rmH}_N^d(R)]_{2-d}  \hookrightarrow [{\rmH}_N^d(G)]_{2-d} ,$$
so that $\ell_A({\rmH}_N^d(R)])=1$, 
since $\ell_A([{\rmH}_N^d(G)]_{2-d})=1$ by Lemma \ref{buch} and ${\depth}~R = d$ by \cite[Theorem 2.1]{HM1}.
Thus $N{\cdot}{\H}_N^d(R) = (0)$, whence $R$ is a Buchsbaum ring.
\end{proof}


\section{An example}
In this section we explore the following example which satisfies the conditions Theorem \ref{Sally} (1) and Theorem \ref{MainTheoremD} (1). The eaxmple is a generalization of an example given by the first author \cite {G}, where the case $\Lambda = \emptyset$ is explored.

Let $m\geq d > 0$ be integers.
Let $\Lambda$ be a subset of $\{1,2,\cdots,m \}$
such that $\Lambda \cap \{1,2,\cdots,d\}=\emptyset$.
Let $$U=k[[X_1,X_2,\cdots,X_m,V,Y_1,Y_2,\cdots,Y_d]]$$
be the formal power series ring over a field $k$ and
let $$\a=(X_1,X_2,\cdots,X_m){\cdot}(X_1,X_2,\cdots,X_m,V)+(V^2-\sum_{i=1}^{d}X_{i}Y_{i}).$$
We put $A=U/\a$ and denote the images of $X_i$, $V$, and $Y_j$ in $A$ by $x_i$, $v$ and $a_j$, respectively. Then ${\dim}~A=d$, since $\sqrt{\a}= (X_1,X_2,\cdots,X_m,V)$.
Let $\m = (x_j \mid 1 \leq j \leq m) + (v) + (a_i \mid 1 \leq i \leq d)$ be the maximal ideal in $A$. We put 
$$I=(a_1, a_2, \cdots,a_d)+(x_{\alpha} \mid \alpha \in \Lambda)+(v)~ \textup{and}~Q=(a_1,a_2, \cdots,a_d).$$ Then ${\m}^2=Q{\m}$, $I^2 = QI + (v^2) \ne QI$,  and $I^3=QI^2$ (cf. Lemma \ref{lemmaex} below), whence
$Q$ is a minimal reduction of both ${\m}$ and $I$, and 
$a_1,a_2,\cdots,a_d$ is a system of parameters for $A$.

We are now interested in the Hilbert coefficients $\e_i's$ of the ideal $I$ as well as the structure of the associated graded ring and the Sally module of $I$. We maintain the same notation as in the previous sections.

We then have he following.

\begin{thm}\label{ex2}
The following assertions hold true.
\begin{itemize}
\item[$(1)$] $A$ is a Cohen-Macaulay local ring with $\dim A = d$.
\item[$(2)$] $S \cong B(-1)$ as graded $T$-modules.
\item[$(3)$] ${\e}_0=m+2$ and ${\e}_1=\sharp \Lambda +2$. Hence, $\e_1 = 2$ but $I^2 \ne QI$, if $\Lambda=\emptyset$. 
\item[$(4)$] ${\e}_2=1$, if $d \geq 2$ and ${\e_i}=0$ for all $3 \leq i \leq d$, if $d \geq 3$.
\item[$(5)$] $G$ is a Buchsbaum ring with  ${\operatorname{depth}}~G=d-1$ and $\ell_A({\H}^{d-1}_N(G))=1$.
\end{itemize}
\end{thm}

We divide the proof of Theorem \ref{ex2} into several steps. Let us  begin with the following. 

\begin{prop}\label{lemmaF}
Let $\fkp = \sqrt{(X_1,X_2,\cdots, X_m, V)}$ in $U$. Then ${\ell}_{U_{\p}}(A_{\p})=m+2$.
\end{prop}

\begin{proof}
Let $\widetilde{k}=k[Y_1, \frac{1}{Y_1}]$ and $\widetilde{U}=U[\frac{1}{Y_1}]$.
We put  $Z_i=\frac{X_i}{Y_1}$ for $1 \leq i \leq m$,
$T_j=\frac{Y_j}{Y_1}$ for $2 \leq j \leq d$, and $W=\frac{V}{Y_1}$.
Then $\widetilde{U}=\widetilde{k}[Z_1,Z_2,\cdots,Z_m,V,T_2,T_3,\cdots,T_d]$
and
$$\a\widetilde{U}=(Z_1,Z_2,\cdots,Z_m){\cdot}(Z_1,Z_2,\cdots,Z_m,W)+(W^2-\sum_{j=2}^dT_jZ_j-Z_1).$$
Because the elements $\{Z_i\}_{1 \leq i \leq m}$, $W$, and $\{T_j\}_{2 \leq j \leq d}$ are algebraically independent over $\widetilde{k}$,
we have
$$ \widetilde{U}/\a \widetilde{U} \cong \overline{U}=
\frac{\widetilde{k}[Z_2,Z_3,\cdots,Z_m,W,T_2,T_3,\cdots,T_d]}{(W^2,Z_2,Z_3,\cdots,Z_m){\cdot}(Z_2,Z_3,\cdots,Z_m,W)},$$
where we substitute $Z_1$ with $W^2-\sum_{j=2}^dT_jZ_j$.
Then the ideal ${\p}\widetilde{U}/K\widetilde{U}$ corresponds to the prime ideal
${P}=(Z_2,Z_3,\cdots,Z_m,W)$.
Thus $\ell_{U_{\p}}(A_{\p})=\ell_{\overline{U}_{P}}(\overline{U}_{P})=m+2$.
\end{proof}

Now we have ${\e}_0(Q)=\ell_{U_{\p}}(A_{\p}) {\cdot} {\e}_0^{A/\fkp A}((Q + \fkp A)/\fkp A)=m+2$ by the associative formula of multiplicity,
because $\fkp=\sqrt{\a}$ and $U/\fkp \cong k[Y_1,Y_2,\cdots,Y_d]$.
On the other hand, $\ell_A(A/Q)=m+2$, since
$$A/Q \cong k[[X_1,X_2,\cdots,X_m,V]]/((X_1,X_2,\cdots,X_m) {\cdot} (X_1,X_2,\cdots,X_m,V)+(V^2)).$$
Hence  ${\e}_0(Q)=\ell_A(A/Q)$, so that  $A$ is a Cohen-Macaulay ring and $\e_0(Q) = m+2$.

\begin{lem}\label{lemmaex}
The following assertions hold true.
\begin{itemize}
\item[$(1)$] ${\m}^2=Q{\m}$, $I^2=QI+(v^2) \neq QI$, and $I^3=QI^2$.
\item[$(2)$] $(a_1,a_2,\cdots,\check{a_i},\cdots,a_d) \cap I^2=(a_1,a_2,\cdots,\check{a_i},\cdots,a_d)I$
            for all $1 \leq i \leq d$.
\item[$(3)$] $(a_{\alpha} \mid  \alpha \in \Gamma) \cap I^n=(a_{\alpha} \mid \alpha \in \Gamma)I^{n-1}$
            for all subsets $\Gamma \subsetneq \{1,2,\cdots,d\}$ and for all integers $n \in {\Z}$.
\item[$(4)$] $(a_1^2,a_2^2,\cdots,a_d^2) \cap I^n=(a_1^2,a_2^2,\cdots,a_d^2)I^{n-2}$
            for all $3 \leq n \leq d+1$.
\end{itemize}
\end{lem}
\begin{proof}
$(1)$
It is routine to check that ${\m}^2=Q{\m}$, and $I^2=QI+(v^2)$. We have $I^3=QI^2$, since $v^3 = 0$. 
Let us check that $v^2 \not\in QI$. 
Suppose $v^2 \in QI$ and write 
$$ v^2=\sum_{i=1}^da_ix_i =\sum_{i=1}^d a_i \xi_i$$
with $\xi_i \in I$.
Then $a_d(x_d-\xi_d) \in (a_1,a_2,\cdots,a_{d-1})$ and so $x_d-\xi_d \in (a_1,a_2,\cdots,a_{d-1})$, because $a_1, a_2, \cdots , a_d$ is a regular sequence.
Hence $x_d \in I$ so that  
$X_d \in \a+(Y_1,Y_2,\cdots,Y_d)U+(X_{\alpha} \mid  \alpha \in \Lambda)U+VU$, which is impossible, because $\Lambda \cap \{1, 2, \cdots , d\} =\emptyset$.

$(2)$
Let $ 1 \leq i \leq d$ be an integer and put $Q_i=(a_1,a_2,\cdots,\check{a_i},\cdots,a_d)$.
Then
\begin{eqnarray*}
Q_i \cap I^2 &=&Q_i \cap (QI+(v^2))\\
             &=&Q_i \cap (Q_iI+a_iI+(v^2))\\
             &=&Q_iI+Q_i \cap [a_iI+(v^2)].
\end{eqnarray*}
Let $\varphi \in Q_i \cap (a_iI+v^2A)$ and write $\varphi=a_i \rho +v^2 \xi$
with $\rho \in I$ and $\xi \in A$.
Then
$\varphi=a_i \rho+\sum_{j=1}^da_jx_j \xi=a_i(\rho+x_i \xi)+\sum_{j \neq i}a_jx_j\xi$. Hence $a_i(\rho+x_i \xi) \in Q_i$ and so $\rho +x_i \xi \in Q_i$; thus $x_i \xi \in I$.
Therefore $\xi \in {\m}=I+(x_{\alpha} \mid \alpha \notin \Lambda)$.
Let $\xi={\xi}'+{\xi}''$
with ${\xi}' \in I$ and ${\xi}'' \in (x_{\alpha} \mid \alpha \notin \Lambda)$.
Notice that 
$x_j \xi=x_j({\xi}'+{\xi}'')=x_j\xi'+x_j\xi''=x_j\xi'$ for all $1 \leq j \leq d$, since $x_j\xi" \in (x_1, x_2, \cdots, x_m)^2=(0)$.
Consequently
$\varphi=a_i(\rho+x_i \xi')+\sum_{j \neq i}a_jx_j\xi' \in Q_iI$, since $\xi' \in I$ and $\rho+x_i \xi' = \rho+x_i \xi \in Q_i$.
Thus $Q_i \cap I^2 \subseteq Q_iI$, so that we have $Q_i \cap I^2=Q_iI$.

$(3)$
Let $\tau=\sharp \Gamma$ and we will prove assertion (3) by descending induction on $\tau$.
Suppose that $\tau=d-1$ and let $\Gamma = \{1, 2, \cdots , \check{i}, \cdots , d \}$ with $1 \leq i \leq d$. If $n \leq 2$, assertion (3) is obvious and follows from assertion (2). Assume that $n \geq 3$ and that our assertion holds true for $n-1$.
Then, since $I^3 = QI^2$, we have 
\begin{eqnarray*}
Q_i \cap I^n&=&Q_i \cap QI^{n-1}\\
            &=&Q_i \cap (Q_iI^{n-1}+a_iI^{n-1})\\
            &=&Q_iI^{n-1}+[Q_i \cap a_iI^{n-1}]\\
            &=&Q_iI^{n-1}+a_i[Q_i \cap I^{n-1}].
\end{eqnarray*}
Since $Q_i \cap I^{n-1}=Q_iI^{n-2}$ by the hypothesis of induction on $n$, we get $$a_i[Q_i \cap I^{n-1}]=a_i[Q_iI^{n-2}] \subseteq Q_iI^{n-1}.$$
Thus $Q_i \cap I^n \subseteq Q_i I^{n-1}$ whence $Q_i \cap I^n = Q_i I^{n-1}$.

We now consider the case where $\tau < d-1$. We assume that $n \geq 2$ and our assertion holds true for $n-1$.
Let $\varphi \in (a_{\alpha} \mid  \alpha \in \Gamma) \cap I^n$ and let $\beta \in \{1,2, \cdots,d\}\, \backslash \, \Gamma$. Then 
$$(a_{\alpha} \mid  \alpha \in \Gamma) \cap I^n \subseteq [(a_{\alpha} \mid \alpha \in \Gamma)+(a_{\beta})] \cap I^n
=[(a_{\alpha} \mid \alpha \in \Gamma)+(a_{\beta})]I^{n-1}$$ by the hypothesis on $\tau$. We write $\varphi=\varphi'+a_{\beta} \rho$  with $\varphi' \in (a_{\alpha}| \alpha \in \Gamma)I^{n-1}$ and  $\rho \in I^{n-1}$.
Then  $a_{\beta} \rho \in (a_{\alpha} \mid  \alpha \in \Gamma)$ and so $\rho \in (a_{\alpha}\mid \alpha \in \Gamma) \cap I^{n-1}$, while $(a_{\alpha} \mid  \alpha \in \Gamma) \cap I^{n-1}=(a_{\alpha} \mid \alpha \in \Gamma)I^{n-2}$ by the hypothesis on $n$.
Hence $\rho \in (a_{\alpha} \mid  \alpha \in \Gamma)I^{n-2}$ so that 
$\varphi \in (a_{\alpha} \mid \alpha \in \Gamma)I^{n-1}$.
Thus
$(a_{\alpha} \mid \alpha \in \Gamma) \cap I^n \subseteq (a_{\alpha} \mid \alpha \in \Gamma)I^{n-1}$ as is claimed. 

$(4)$
We put $J=(a_1^2,a_2^2,\cdots,a_d^2)$.
Assume that $J \cap I^n \neq JI^{n-2}$ for some $3 \leq n \leq d+1$ and choose $d$ as small as possible among such counterexamples. Hence $d \geq 2$. Let $\varphi \in J \cap I^n$ such that $\varphi \not\in JI^{n-2}$.

We  begin with the following.

\begin{claim}\label{claimE}
$$I^{d+1}=JI^{d-1}+a_1a_2\cdots a_dI+\sum_{i=1}^da_1a_2 \cdots \check{a_i} \cdots a_d v^2A.$$
\end{claim}

\begin{proof}[Proof of Claim \ref{claimE}]
Since $I^2=QI+(v^2)$ and $I^3=QI^2$,
we have $$I^{d+1}=Q^{d-1}I^2=Q^{d-1}(QI+(v^2))=Q^dI+v^2Q^{d-1}.$$ On the other hand, because $$Q^d = JQ^{d-2} + (a_1a_2\cdots a_d)~~\operatorname{and}~~Q^{d-1} = JQ^{d-3} + \sum_{i=1}^da_1a_2 \cdots \check{a_i} \cdots a_dA,$$ we get 
$$Q^dI=JQ^{d-2}I+a_1a_2 \cdots a_d I \subseteq JI^{d-1}+a_1a_2 \cdots a_dI$$
and 
\begin{eqnarray*}
v^2Q^{d-1} &=& v^2JQ^{d-3}+v^2(\sum_{i=1}^da_1a_2 \cdots \check{a_i} \cdots a_dA) \\
&\subseteq& JI^{d-1} 
+ \sum_{i=1}^d a_1a_2 \cdots \check{a_i} \cdots a_d v^2A,
\end{eqnarray*}
(notice that $v \in I)$.
Hence
$I^{d+1} \subseteq JI^{d-1}+a_1a_2\cdots a_dI+\sum_{i=1}^da_1a_2 \cdots \check{a_i} \cdots a_d v^2A$.
\end{proof}
Suppose that $n=d+1$. Then by Claim \ref{claimE} we may write $$\varphi=\varphi'+a_1a_2 \cdots a_d \eta +\sum_{i=1}^d c_i a_1 a_2 \cdots \check{a_i} \cdots a_d v^2$$
 with $\varphi' \in JI^{d-1}$, $\eta \in I$, and $c_i \in A$. Since  $v^2=\sum_{i=1}^da_ix_i$, we see $$\sum_{i=1}^d c_i a_1 a_2 \cdots \check{a_i} \cdots a_d v^2  \equiv a_1a_2\cdots a_d \left( \sum_{i=1}^dc_ix_i \right)~{\mod}~J$$ whence 
$$a_1a_2 \cdots a_d(\eta+\sum_{i=1}^dc_ix_i) \equiv a_1a_2 \cdots a_d\eta+ \sum_{i=1}^d c_i a_1 a_2 \cdots \check{a_i} \cdots a_d v^2 \equiv 0~{\mod}~ J,$$
because
$$\varphi=\varphi'+a_1a_2 \cdots a_d \eta +\sum_{i=1}^d c_i a_1 a_2 \cdots \check{a_i} \cdots a_d v^2 \in J.$$ 
Hence $\eta+\sum_{i=1}^dc_ix_i \in Q$ because $a_1, a_2, \cdots, a_d$ is a regular sequence in $A$,  so that we have 
$$\sum_{i=1}^dc_ix_i \in I =(a_i \mid 1 \leq i \leq d) + (x_{\alpha} \mid \alpha \in \Lambda) +(v).$$
Because $\{x_i\}_{1 \leq i \leq m}$, $v$, and $\{a_i\}_{1 \leq i \leq d}$ is a minimal basis of the maximal ideal $\m $ of $A$ and $\Lambda \cap \{1, 2, \cdots , d\} = \emptyset$, this forces $c_i \in \m$ for all $1 \leq i \leq d$.
We write $c_i = c_i'+c_i"$ with $c_i' \in Q$ and $c_i" \in (x_1,x_2,\cdots,x_m,v)$. Then, since $(x_1,x_2,\cdots,x_m,v){\cdot}(x_1,x_2,\cdots,x_m)=(0)$, we have $c_i"x_i =0$ and so $$c_ix_i=c_i'x_i+c_i"x_i=c_i'x_i \in Q$$ because  $c_i' \in Q.$
Consequently, since $\eta+\sum_{i=1}^dc_ix_i \in Q$, we have
$$\eta \equiv \eta+\sum_{i=1}^dc_i'x_i = \eta+\sum_{i=1}^dc_ix_i \equiv 0 ~{\mod}~ Q.$$
Hence $\eta \in Q$ and so $$a_1a_2 \cdots a_d \eta \in Q^{d+1}=(a_1^2,a_2^2,\cdots,a_d^2)Q^{d-1} \subseteq JI^{d-1}.$$
On the other hand we have $c_i"v^2=0$ since $c_i" \in (x_1,x_2,\cdots,x_m,v)$, so that $c_iv^2=c_i'v^2+c_i"v^2=c_i'v^2 \in Q^2$ because $c_i', v^2 \in Q$. Hence  
$$c_ia_1a_2 \cdots \check{a_i} \cdots a_dv^2 = a_1a_2 \cdots \check{a_i} \cdots a_d{\cdot}c_i'v^2 \in Q^{d+1} \subseteq JI^{d-1}$$
for all $1 \leq i \leq d$, so that 
$$\varphi=\varphi'+a_1a_2 \cdots a_d \eta +\sum_{i=1}^d c_i a_1 a_2 \cdots \check{a_i} \cdots a_d v^2 \in JI^{d-1}, $$which is a contradiction. Thus $3 \leq n \leq d$.

We put $\overline{A}=A/(a_d)$ and $\overline{I}=I/(a_d)$. For each $x \in A$ let $\overline{x}$ denote the image of $x$ in $\overline{A}$. We then have, by the minimality of $d$, that 
$$(\overline{a_1^2},\overline{a_2^2},\cdots,\overline{a_{d-1}^2}) \cap \overline{I}^n
=(\overline{a_1^2},\overline{a_2^2},\cdots,\overline{a_{d-1}^2})\,\overline{I}^{n-2}$$ for all $3 \leq n \leq d$. Hence $\overline{\varphi} \in (\overline{a_1^2},\overline{a_2^2},\cdots,\overline{a_{d-1}^2})\overline{I}^{n-2}$, so that 
$$\varphi \in (a_1^2,a_2^2, \cdots,a_{d-1}^2)I^{n-2}+[(a_d) \cap I^n].$$
Since  $(a_d) \cap I^n=a_dI^{n-1}$ by assertion $(3)$, we have $\varphi=\varphi'+a_d \xi$ for some $\varphi' \in (a_1^2,a_2^2,\cdots,a_{d-1}^2)I^{n-2}$ and $\xi \in I^{n-1}$; hence $a_d \xi \in J$, because $\varphi, \varphi' \in J$. We write $a_d \xi=\sum_{i=1}^da_i^2 \xi_i$
with $\xi_i \in A$.
Then $a_d(\xi-a_d\xi_d) \in (a_1^2,a_2^2,\cdots,a_{d-1}^2)$, so that $\xi-a_d\xi_d \in (a_1^2,a_2^2,\cdots,a_{d-1}^2)$.
Consequently 
$$\overline{\xi} \in (\overline{a_1^2},\overline{a_2^2},\cdots,\overline{a_{d-1}^2})\cap \overline{I}^{n-1} =(\overline{a_1^2},\overline{a_2^2},\cdots,\overline{a_{d-1}^2})\,\overline{I}^{n-3}$$ by the minimality of $d$.
 Hence 
$$\xi \in (a_1^2,a_2^2,\cdots,a_{d-1}^2)I^{n-3}+[(a_d) \cap I^{n-1}].$$
However, since $(a_d) \cap I^{n-1}=a_dI^{n-2}$ by assertion $(3)$,
we have $$a_d \xi \in a_d(a_1^2,a_2^2,\cdots,a_{d-1}^2)I^{n-3}+a_d^2I^{n-2} \subseteq JI^{n-2},$$
whence $\varphi=\varphi'+a_d\xi \in JI^{n-2}$, which is the required contradiction.
Thus $$J \cap I^n=JI^{n-2}$$ for all $3 \leq n \leq d+1,$ as we wanted.
\end{proof}

We are now in a position to complete the proof of Theorem \ref{ex2}.

\begin{proof}[Proof of Theorem \ref{ex2}.]
We have  $\ell_A(I^2/QI)=1$, since ${\m}v^2 \subseteq QI$  (recall that $I^2 \neq QI$ and $I^2 = QI + (v^2)$ by Lemma \ref{lemmaex} (1)). Because $I^3 = QI^2$, by Theorem \ref{Sally} we have $S \cong B(-1)$ as graded $T$-modules, so that $\e_1 = \e_0 - \ell_A(A/I) + 1$, $\e_2 = 1$ if $d \geq 2$, and $\e_i = 0$ for all $3 \leq i \leq d$ if $d \geq 3$. Because $\ell_A(A/I) = m - \sharp{\Lambda} +1$ and $\e_0 = m + 2$, we get $\e_1 = \sharp{\Lambda} + 2$; hence $\e_1=2$ if $\Lambda = \emptyset$. 

Notice that $G$ is not a Cohen-Macaulay ring. In fact, $Q \cap I^2  \ne QI$ (recall that $I^2 \subseteq Q$ since ${\m}^2 =Q{\m}$). The ring $G$ is Buchsbaum by Lemma \ref{lemmaex} $(1)$, $(2)$, and $(4)$ and \cite[Proposition 9.1]{G} and so the facts that ${\H}_N^{d-1}(G)= [{\H}_N^{d-1}(G)]_{1-d}$ and $\ell_A([{\H}_N^{d-1}(G)]_{1-d})=1$ follow  by induction on $d$ similarly as in the proof of Lemma \ref{buch} and Theorem \ref{MainTheoremD}.
\end{proof}


\addcontentsline{toc}{section}{references}

\begin{thebibliography}{99}

\bibitem{BS} M. P. Brodmann and R. Y. Sharp, \textit{Local Cohomology}, Cambridge studies in advanced mathematics, vol. $\bf60$, Cambridge University Press, Cambridge, 1998.

\bibitem{CPVP} A. Corso, C. Polini and M. Vaz Pinoto, \textit{Sally modules and associated graded rings}, Comm. Algebra, $\bf26$, 1998, 2689--2708.

\bibitem{G} S. Goto, \textit{Buchsbaumness in Rees Algebras Associated to ideals of Minimal Multiplicity}, J. Algebra, Vol $\bf213$, 1999, 604--661


\bibitem{H} C. Huneke, \textit{Hilbert functions and symbolic powers}, Michigan Math. J., Vol $\bf34$, 1987, 293--318

\bibitem{HM1} S. Huckaba and T. Marley, \textit{Depth properties of Rees Algebras and associated graded rings.}, J. Algebra, 156, 1993, no.1, 259--271


\bibitem{Na} M. Narita, \textit{A note on the coefficients of Hilbert characteristic functions in semi-regular rings}, Proc. Cambridge Philos. Soc. $\bf59$, 1963, 269--275

\bibitem{No} D. G. Northcott, \textit{A note on the coefficients of the abstract Hilbert function}, J. London Math. Soc., Vol $\bf35$, 1960, 209--214


\bibitem{RR} L. J. Ratliff and D. Rush, \textit{Two notes on reductions of ideals}, Indiana Univ. Math. J. $\bf27$ (1978), 929-934.

\bibitem{R} M. E. Rossi, \textit{A bound on the reduction number of a primary ideal}, Proc. Amer. Math. Soc. 128 (2000), no. $\bf5$, 1325--1332.


\bibitem{S1} J. D. Sally, \textit{Cohen-Macaulay local rings of maximal embedding dimension}, J. Algebra, $\bf56$, 1979, 168--183

\bibitem{S2} J. D. Sally, \textit{Tangent cones at Gorenstein singularities}, Composito Math. $\bf40$, 1980, 167--175.

\bibitem{S3} J. D. Sally, \textit{Hilbert coefficients and reduction number 2}, J. Alg. Geo. and Sing. $\bf1$, 1992, 325--333.

\bibitem{V} W. V. Vasconcelos, \textit{Hilbert Functions, Analytic Spread, and Koszul Homology}, Contemporary Mathematics, Vol $\bf159$, 1994, 410--422.

\bibitem{VP} M. Vaz Pinoto, \textit{Hilbert functions and Sally modules}, J. Algebra, $\bf192$, 1977, 504--523.
\end{thebibliography}
\end{document}